\documentclass[a4paper,11pt]{article}

\usepackage[T1]{fontenc}
\usepackage{fullpage}%
\usepackage[utf8]{inputenc}%
\usepackage[main=english]{babel}%
\usepackage{graphicx}%
\usepackage{url}%
\usepackage{listings}%
\usepackage{amsmath}%
\usepackage{amssymb}%
\usepackage{amsfonts}%
\usepackage{stmaryrd}%
\usepackage{amsthm}%
\usepackage{array}%
\usepackage{csquotes}%
\usepackage{amsmath}
\usepackage{amssymb}
\usepackage{pifont}
\usepackage{xcolor} 
\usepackage{framed} 
\usepackage{mdframed}
\usepackage{hyperref}

\def\RR{{\mathbb R}}
\def\NN{{\mathbb N}}
\def\ZZ{{\mathbb Z}}
\def\PP{{\mathbb P}}

\def\sig{{\sigma}}
\def\lbd{{\lambda}}
\def\usup{{u^+}}

\theoremstyle{definition}
\newmdtheoremenv{theo}{Theorem}
\newmdtheoremenv{coro}{Corollary}
\newmdtheoremenv{defi}{Definition}
\newmdtheoremenv{prop}{Proposition}

\newtheorem{property}{Property}

\begin{document}

\title{An example of failure of stochastic homogenization for viscous Hamilton-Jacobi equations without convexity}
\author{William M. Feldman \and Jean-Baptiste Fermanian \and Bruno Ziliotto}
\date{}
\maketitle
\begin{abstract}
We give an example of the failure of homogenization for a viscous Hamilton-Jacobi equation with non-convex Hamiltonian.
\end{abstract}
\section*{Introduction}
We consider the problem of stochastic homogenization for viscous Hamilton-Jacobi equations. Let $H : \RR^n \times \RR^n  \times \Omega \to \RR$ ($n\geq 1$) which is a stationary ergodic random field on the probability space $(\Omega, \mathcal{F}, \PP)$. To that function $H$, named Hamiltonian, we associate an equation called Hamilton-Jacobi equation
\begin{equation}\label{eq_intro}
\left\{ \begin{array}{ll}
            \partial_t u^{\varepsilon} (t,x,\omega) + H\left( \nabla u^{\varepsilon}(t,x,\omega), \frac{x}{\varepsilon}, \omega\right) - \varepsilon \Delta u^{\varepsilon} = 0  & \; \text{in} \; (0, + \infty ) \times \RR^n \\
            u^{\varepsilon}(0,x)  = u_0(x) & \; \text{in} \; \RR^n
          \end{array} \right.
\end{equation}
where $\varepsilon > 0$.

The question of stochastic homogenization is to study the convergence properties of $u^\varepsilon$, as $\varepsilon$ goes to $0$.  When $H$ is periodic (Lions, Papanicolaou and Varadhan \cite{Lions}), when $H$ is convex (Souganidis \cite{Souganidis} and Rezakhanlou and Tarver \cite{Rezakhanlou}), or when $H$ is positively homogeneous in the gradient variable and the law of $H$ satisfies a finite range of dependence condition (Armstrong and Cardaliaguet 2015 \cite{Cardaliaguet}), $u^\varepsilon$ converges $\PP$-almost surely and locally uniformly in $(t,x)$ to the unique solution of a system of the form :
\begin{equation}\label{eq_intro2}
\left\{ \begin{array}{ll}
            \partial_tu(t,x,\omega) + \bar{H}( \nabla u(t,x)) = 0  & \; \text{in} \; (0, + \infty ) \times \RR^n \\
            u(0,x)  = u_0(x) & \; \text{in} \; \RR^n
          \end{array} \right.
\end{equation}
where $\bar{H}$ is called the effective Hamiltonian. In the non-viscous case, dimension at least $2$ and general stationary ergodic environment, homogenization may fail without the convexity assumption (work of the third author \cite{ziliotto}).

The question of the convergence for the viscous Hamilton-Jacobi equation without convexity has remained open.  Several recent works have given examples of non-convex Hamiltonians where homogenization holds \cite{Cardaliaguet,Davini,KosyginaYilmazZeitouni,yilmaz2019}.   In this note we extend the work of the third author~\cite{ziliotto} and show that homogenization may fail for \eqref{eq_intro}, with all the standard assumptions of the literature except convexity.
\begin{theo}
There exists a Hamiltonian $h: \RR^2 \to \RR$, Lipschitz continuous and coercive, and a stationary and ergodic random field $c: \RR^2 \to \RR$, bounded and Lipschitz regular, such that homogenization fails for \eqref{eq_intro} with $H(p,x) = h(p) - c(x)$ and initial data $u(0,x) = 0$.
\end{theo}
Our example is in dimension $n=2$, and the construction is for a specific form of Hamiltonian.  However, as in the work of the first author and Souganidis~\cite{Feldman}, the example can be extended to the following generality:  Suppose that $n \geq 2$ and $h: \RR^n \to \RR$ is Lipschitz continuous and coercive and has a strict saddle point at some $p_0 \in \RR^n$.  Then there exists a stationary and ergodic random field $c: \RR^n \to \RR$, bounded and Lipschitz regular, such that homogenization fails for \eqref{eq_intro} with $H(p,x) = h(p) - c(x)$ and initial data $u(0,x) = p_0 \cdot x$.

\subsection*{Literature}
The problem of homogenization for Hamilton-Jacobi equation in random environments has been studied quite intensively over the last decade, following earlier pioneering works.  Generally speaking the first order case has been studied first, the viscous case is usually more difficult and analogous results follow later.  This began with the work of Souganidis \cite{Souganidis} and of Rezakhanlou and Tarver \cite{Rezakhanlou} who independently proved homogenization in the case of $H$ convex in general stationary ergodic random media.  The literature is vast, see for example \cite{LSV1}, \cite{Gao1}, and \cite{KosyginaYilmazZeitouni} for more thorough reviews of the history of the subject (especially with focus on the viscous case).

In the viscous case homogenization results have been proven, all in the convex setting, by Lions and Souganidis \cite{LSV1,LSV2}, Kosygina, Rezakhanlou and Varadhan \cite{KRV}, Kosygina and Varadhan \cite{KV}, Armstrong and Souganidis \cite{ASunbounded}, Armstrong and Tran \cite{ArmstrongTran}.

There has been a lot of interest to push these homogenization results further to truly non-convex settings.

 In the first order $1$-dimensional case homogenization holds without any convexity assumption.  This was proven in a series of works started by Armstrong, Tran and Yu~\cite{ArmstrongTranYu1,ArmstrongTranYu2} (for Hamiltonians in separated form) and completed by Gao~\cite{Gao1} (general coercive Hamiltonians in $d=1$).  These techniques generalize to some higher dimensional problems with a special structure, see Gao~\cite{Gao2}.

 In the viscous case in $1$-d several special families of nonconvex Hamiltonians have been considered.  Davini and Kosygina~\cite{Davini2017} considered Hamiltonians with one or more ``pinning points" $p$ where $H$ is constant almost surely, and convex in between.  Y{\i}lmaz and Zeitouni \cite{yilmaz2019} (discrete case) and Kosygina, Y{\i}lmaz, and Zeitouni \cite{KosyginaYilmazZeitouni} (continuous case) have proven homogenization for a special nonconvex Hamiltonian $H(p,x) = \frac{1}{2}|p|^2-c|p|+V(x,\omega)$.

 A fundamental limitation to these efforts was discovered by the third author \cite{ziliotto}, who found an example of a first order non-convex Hamilton-Jacobi equation for which homogenization fails.  Following that, the first author and Souganidis \cite{Feldman} found that a similar example will cause non-homogenization for any Hamiltonian having a strict saddle point.  In the current work we resolve the same question in the viscous case, showing that homogenization does not hold, by adapting the construction of \cite{ziliotto}. Like in \cite{ziliotto}, the idea of the example comes from a zero-sum differential game.
Our construction is also simpler than in the former paper, and is symmetric in players' roles. Thus, this paper illustrates once more how dynamic games can be helpful in PDE problems, as already emphasized in \cite{KS06,KS10,IS11, Z19}.

\section{Preliminaries}
The Hamiltonian will always satisfy the following properties :
\begin{itemize}
  \item $p \to H(p,.,.)$ has superlinear growth
  \item $p \to H(p,.,.)$ is Lipschitz
  \item The law of $\omega \to H(.,.,\omega)$ is invariant by $\RR^d$ translations
\end{itemize}
Under these assumptions, the system (\ref{eq_intro}) admits a unique solution of viscosity $u^\varepsilon$ which is measurable with respect to $\omega$.

As will be made clear below, actually $H(p,x,\omega) = h(p) - c_\omega(x)$, and $h$ can be chosen as an arbitrary coercive Hamiltonian in the region $\|p\|_\infty \geq 2$.  In particular superlinear growth is not a requirement of the example, it is only emphasized as it is a usual assumption in the literature proving positive results on homogenization.

 \section{The weight function $c_\omega$}

Let $E$ be the set of 1-Lipschitz mappings from $\RR^2$ to $[-\frac{1}{2},\frac{1}{2}]$. Let us build a probability measure on $E$ in the following way.
Let $(T_k)_{k \geq 1}$ the sequence defined by $T_k = 2^k$. Let $(X^j_{k,l,m})_{(j,k,l,m) \in \{1,2\} \times \NN^* \times \ZZ^2}$ a sequence of independent random variables defined on a probability space $( \Omega, \mathcal{F}, \mathbb{P})$ such that for all $(j,k,l,m) \in \{1,2\} \times \NN^* \times \ZZ^2$, $X^j_{k,l,m}$ follows a Bernoulli of parameter $T_k^{-2}$.

Fix some $\lambda$, $\mu >0$. For $\omega \in \Omega$, the mapping $c_\omega : \RR^2 \to [-\frac{1}{2},\frac{1}{2}] \in E$ is built in three phases.

\paragraph{Phase 1}

The mapping  $c^1_\omega : \RR^2 \to [-\frac{1}{2},\frac{1}{2}] $ is built through the following step-by-step procedure.
\begin{itemize}
  \item Step $k=0$: take $c^1_\omega :=0$ as the initial distribution of weights.
  \item Step $k \geq 1$ : for each $(l,m) \in \ZZ^2$ such that $X^1_{k,l,m} =1$, consider the rectangle centered on $(l,m)$ with length $\lambda T_k +1$ and width $\mu+1$, which shall be called \textit{horizontal rectangle of length $\lambda T_k +1$}. For each $x\in \RR^2$  that lies in the rectangle, set $c^1_\omega(x) := -\frac{1}{2}$.
\end{itemize}
 At the end of phase 1 we have a map $c^1_\omega : \RR^2 \to [-\frac{1}{2},\frac{1}{2}] $.

 \paragraph{Phase 2}

 The mapping  $c^2_\omega : \RR^2 \to [-\frac{1}{2},\frac{1}{2}]$ is built through the following step-by-step procedure.
\begin{itemize}
  \item Step $k=0$: take $c^2_\omega := c^1_\omega$ as the initial distribution of weights.
  \item Step $k \geq 1$ : for each $(l,m) \in \ZZ^2$ such that $X^2_{k,l,m} =1$, consider the vertical rectangle centered on $(l,m)$ with length $\lambda T_k +1$ and width $\mu +1$, which shall be called \textit{vertical rectangle of length $\lambda T_k +1$}. For each $x\in \RR^2$  that lies in the rectangle, we distinguish three cases :
      \begin{itemize}
        \item If $x$ lies in a horizontal rectangle of size $\lambda T_{k'} +1$ with $k' > k$, $c^2_\omega(x)$ is not modified.
        \item If $x$ lies in a horizontal rectangle of size $\lambda T_{k'} +1$ with $k' = k$, $c^2_\omega(x)=0$.
        \item Otherwise, set $c^2_\omega(x) = \frac{1}{2}$.
      \end{itemize}
\end{itemize}
The idea behind this phase is to keep the largest rectangle when two rectangles of different orientations intersect each other. The intersection "turns vertical" if the vertical rectangle's length is strictly larger, and "turns horizontal" if it is strictly smaller. If the rectangles have the same length, the intersection has a null weight.

\begin{figure}[t]
\centering
\includegraphics[width=.7\linewidth]{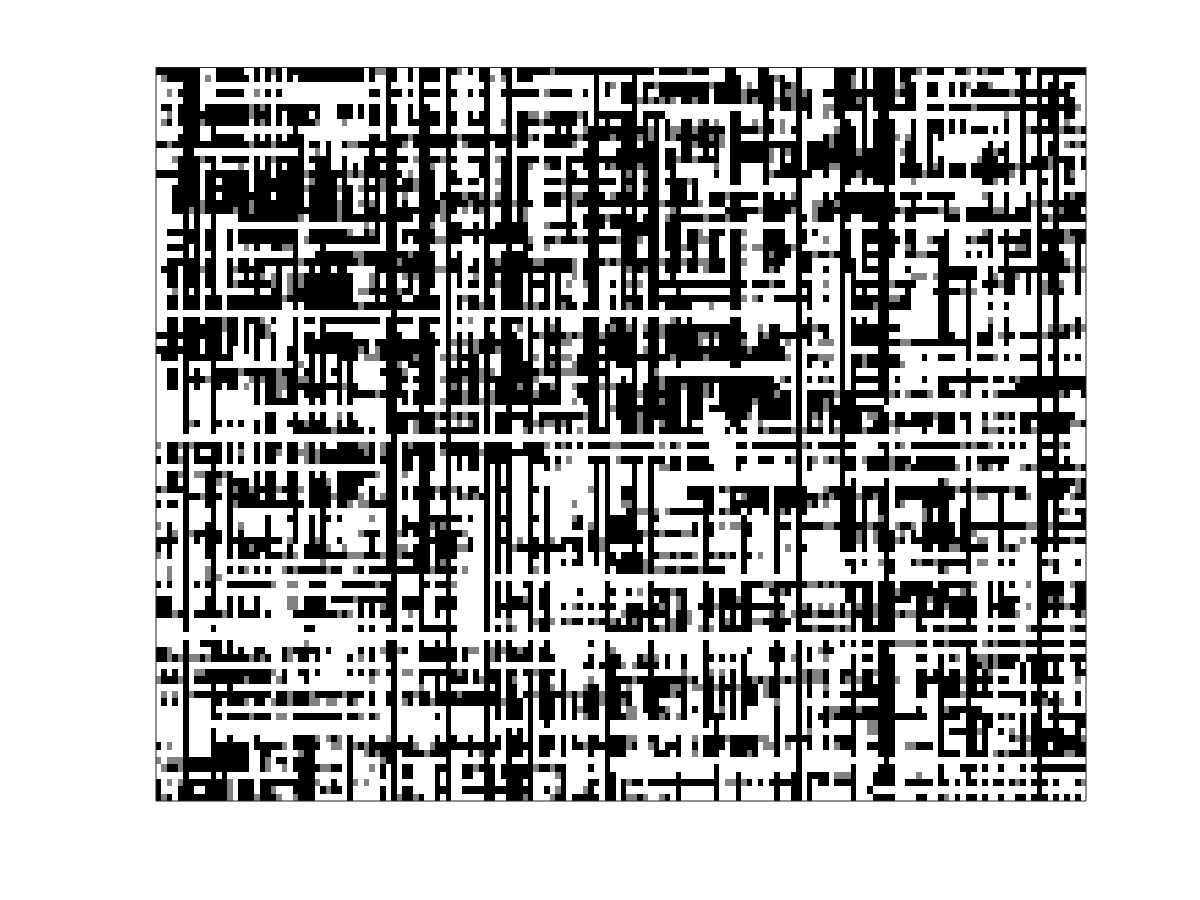}
\caption{An approximate sample environment $c^2_\omega$.}
\end{figure}

\paragraph{Phase 3}
In this phase we just make $c^2_\omega$ a Lipschitz function. Define $c_\omega : \RR^2 \to [-\frac{1}{2}, \frac{1}{2}]$ by:
\begin{equation*}
  c_\omega(x) = \left\{ \begin{array}{cc}
  \max \left(\underset{y \in \RR^2}{\inf} \{ c^2_\omega(y) + |x-y| \}, 0 \right) & \; \text{if} \; c^2_\omega(x) \geq 0 \\
                  \min \left(\underset{y \in \RR^2}{\sup} \{ c^2_\omega(y) -  |x-y| \}, 0 \right) & \; \text{if} \; c^2_\omega(x) \leq 0
                \end{array} \right.
\end{equation*}
Because $c^2_\omega$ takes value in $\{ -\frac{1}{2}, 0 , \frac{1}{2} \}$ we have $\textup{Sgn}(c^2_\omega) = 2c^2_\omega$. So a compact expression of $c_\omega$ is :
$$c_\omega(x) := 2c^2_\omega(x) \max \left( \underset{y \in \RR^2}{\inf} \{2c^2_\omega(x) c^2_\omega(y) + |x-y| \}, 0 \right). $$
It is straightforward from the above expression that $c_{\omega}$ is Lipschitz.
\\

We will name \textit{complete vertical rectangle of length $\lambda T_k +1$ } a vertical rectangle of length $\lambda T_k+1$ which for each $x\in \RR^2$  that lies in the rectangle at a distance larger than $\frac{1}{2}$ of the edges, we have $c_\omega(x) = \frac{1}{2}$. We will say the same for a horizontal rectangle such that $c_\omega(x) = -\frac{1}{2}$.

\section{Main result}
Let $\eta \in (0,1/2)$ and $\lambda, \mu>0$ such that
\begin{equation} \label{cond_eta}
\min(\lambda,\mu) \geq \frac{16}{\eta}.
\end{equation}
Define
\[ H(p,x,\omega ) = - c_\omega(x) + |p_2| - |p_1| + \max \left( ||p||_\infty -2,0 \right) ^q\]
 with $q>1$.  {Note that the definition of $H$ in $\|p\|_\infty \geq 2$ is basically arbitrary.  With our specific choice it holds that, for all $\omega \in \Omega$, $H$ has superlinear growth in $p$, uniformly in $x$ and $\omega$. }

 For $\varepsilon > 0$, consider the following Hamilton-Jacobi equation :
\begin{equation}\label{homo}
\left\{ \begin{array}{ll}
          \partial_t u(t,x,\omega) + H( \nabla u(t,x,\omega) , \frac{x}{\varepsilon}, \omega) - \varepsilon\Delta u = 0 & \text{in} \; (0,\infty )  \times \RR^2  \\
          u(0,x) = 0 & \text{in} \; \RR^2.
         \end{array}
\right.
\end{equation}

\begin{theo}\label{resultat} Let $u^\varepsilon$ be the solution of (\ref{homo}). Then the following holds almost surely:
\begin{align*}
 \underset{\varepsilon \to 0}{\liminf}\, u^\varepsilon(1,0,\omega) \leq -\frac{1}{2} + \eta\\
 \underset{\varepsilon \to 0}{\limsup}\, u^\varepsilon(1,0,\omega) \geq  \frac{1}{2} - \eta.
\end{align*}
Consequently, there is no stochastic homogenization for the above Hamilton-Jacobi equation.
\end{theo}

\section{Game-theoretic interpretation} \label{game}
The Hamilton-Jacobi equation of the previous subsection can be associated to a zero-sum stochastic differential game. We give here an informal description of this game, in which technical details are avoided, notably concerning the definition of strategies. For simplicity, we consider the case where $\lambda$ and $\mu$ are large, and $\eta$ is small.

Let $\omega \in \Omega$ and $T>0$. The game starts at the origin $(0,0)$, and has a duration $T$. Player 1 (resp. 2) aims at minimizing (resp. maximizing) the total cost between time 0 and time $T$, given by $\int_{0}^T c_{\omega}(x(t)) dt$, where $x(t)$ is the state of the game at stage $t$.

The dynamics of the state is such that if Player 1 chooses a control $a \in [-1,1]$ and Player 2 chooses a control $b \in [-1,1]$, then the state moves according to the vector $(b,a)$, with a Brownian motion perturbation. Thus, Player 1 controls the vertical component of the state, while Player 2 controls the horizontal component. The normalized value of the game with duration $T$ coincides with $u^{\frac{1}{T}}(1,0,\omega)$, where $u^{\frac{1}{T}}$ is the solution of the system (\ref{homo}), for $\epsilon=1/T$.

For all $x \in \RR$, denote by $\lfloor x \rfloor$ the integer part of $x$. The construction of $(c_{\omega})_{\omega \in \Omega}$ has been made such that for all $\delta>0$, there exist two positive probability events $\Omega_1$ and $\Omega_2$ such that the following properties hold:
\begin{property}
For all $\omega \in \Omega_1$, there exists a sequence $(n_k(\omega))$ going to infinity such that for all $k \geq 1$, there exists a complete horizontal rectangle of length $\lambda T_{n_k(\omega)}+1$ whose center is at a distance smaller or equal to $\lfloor \delta T_{n_k(\omega)} \rfloor$ from the origin.
\end{property}
\begin{property}
For all $\omega \in \Omega_2$, there exists a sequence $(n'_k(\omega))$ going to infinity such that for all $k \geq 1$, there exists a complete vertical rectangle of length $\lambda T_{n_k(\omega)}+1$ whose center is at a distance smaller or equal to $\lfloor \delta T_{n'_k(\omega)} \rfloor$ from the origin.
\end{property}
Let $k \geq 1$ and $\omega \in \Omega_1$. Consider the game with duration $T_{n_k(\omega)}$. Making use  of vertical controls, Player 1 can force the state to go close to the center of the complete horizontal rectangle, within a length of time smaller or equal to $\lfloor \delta T_{n_k(\omega)} \rfloor$. Then, he can force it to stay most of the time in the rectangle until the end of the game, by making use only of vertical controls. The Brownian motion will push the state outside the rectangle from time to time, but since the width of the rectangle is large, this will not happen often. Player two can push horizontally with unit speed, but it will take time approximately $\lambda T_k \gg T_k$ before the state leaves the rectangle.  Thus, for $\delta$ small enough, the normalized value of the game with duration $T_{n_k(\omega)}$ is close to $-1/2$.

By symmetry, for $k \geq 1$ and $\omega \in \Omega_2$, the normalized value of the game with duration $T_{n'_k(\omega)}$ is close to $1/2$. By ergodicity of the environment, these arguments prove the theorem, and the next section is dedicated to the formal proof.

\section{Proof of the theorem}
\subsection{Presentation of the proof}
For simplicity, we consider the following Hamilton-Jacobi equation:
\begin{equation}\label{homo2}
\left\{ \begin{array}{ll}
          \partial_t u(t,x,\omega) + H( \nabla u(t,x,\omega) , x, \omega) - \Delta u = 0 & \text{in} \; (0,\infty ) \times \RR^2  \\
          u(0,x) = 0 & \text{in} \; \RR^2.
         \end{array}
\right.
\end{equation}
The solution $u$ of the above equation satisfies $u(t,x,\omega)=u^{\varepsilon}(\varepsilon t, \varepsilon x, \omega)/\varepsilon$, for all $(t,x,\omega) \in (0,\infty) \times \mathbb{R}^2 \times \Omega$. Thus, it is equivalent to study the limit of $u^\varepsilon(1,0,\omega)$ and of $\varepsilon u\left(\frac{1}{\varepsilon},0, \omega\right)$, as $\varepsilon$ vanishes.

\subsection{Supersolution of the Hamilton-Jacobi equation}
\subsubsection{The event}
Let ${\delta}> 0$ and $B_k$ the event "there exists a center of a complete horizontal rectangle of length $\lambda T_k +1$ at a distance smaller or equal to $\lfloor {\delta} T_k \rfloor$ from the origin".

\begin{prop}
There exists a positive probability event $\Omega_1 \subset \Omega$ such that for all $\omega \in \Omega_1$, the events $(B_k)_{k\geq 1 }$ occur infinitely often.

\end{prop}

\begin{proof}
  A sufficient condition for $B_k$ to be realized is that the two following events, respectively denoted $C_k$ and $D_k$, are realized:

  \begin{itemize}
    \item At step $k$ of Phase 1, a point at a distance smaller than $\lfloor{\delta} T_k \rfloor $ has been selected by the Bernoulli random variable.
    \item The horizontal rectangle centered on this point is complete, that is, it is not intersected by a vertical rectangle of length larger or equal than $\lambda T_k +1 $.
  \end{itemize}
  We have
  \begin{align*}
   \PP(C_k) &\geq \PP(\overline{\underset{l,m \leq \lfloor {\delta} T_k \rfloor}{\bigcap} \overline{X^1_{k,l,m}}  }) &\\
     & \geq 1 - \prod_{l,m \leq \lfloor {\delta} T_k \rfloor}(1 - \PP(X^1_{k,l,m})) & \text{by independence} \\
     & \geq 1 - (1 - T_k^{-2} )^{\lfloor {\delta} T_k \rfloor^2} \\
     & \underset{k \to +\infty}{\longrightarrow} 1 - e^{-{\delta}^2}.
  \end{align*}

   For $k' \geq k$, the probability that no vertical rectangle of length $\lambda T_{k'}$ intersects the horizontal rectangle is bounded below by the probability that no point in the rectangle with the same center, length $(\lambda T_{k'} +1 + \mu +1)$, and width $(\lambda T_k +1 +\mu +1) $ has been selected by the Bernoulli random variable at step $k'$ during Phase 2, that is,
  $$(1-T_{k'}^{-2})^{(\lambda T_{k'} +\mu +2)(\lambda T_k + \mu+2)}.$$
  Thus
  $$\PP(D_k | C_k) \geq \prod_{k' \geq k}(1-T_{k'}^{-2})^{(\lambda T_{k'}+\mu +2)(\lambda T_k + \mu+2)}.$$
 We have the asymptotic equivalence
  \begin{align*}
    (\lambda T_{k'}+\mu+2)(\lambda T_k + \mu+2)\ln(1 -T_{k'}^{-2})  & \underset{k' \to +\infty}{\sim} -\lambda^2T_{k'}^{-1}T_k.
  \end{align*}
 So
 \begin{align*}
   \sum_{k' \geq k} (\lambda T_{k'}+\mu+2)(\lambda T_k + \mu+2)\ln(1 -T_{k'}^{-2}) &  \underset{k \to +\infty}{\sim} -\lambda^2T_k \sum_{k' \geq k}T_{k'}^{-1} \\
    &  \underset{k \to +\infty}{\sim} -\lambda^2T_k T_{k-1}^{-1} \\
    & \underset{k \to +\infty}{\longrightarrow} -2\lambda^2.
 \end{align*}
 That gives us a minoration of $\liminf \PP(B_k)$
  \begin{align*}
 \displaystyle{\liminf_{k \to \infty} \PP(B_k)} &  \geq  \displaystyle{\liminf_{k \to \infty} \PP(C_k)\PP(D_k | C_k)}  \\
     & \geq (1 - e^{-{\delta}^2})e^{-2\lambda^2} \\
     & > 0.
  \end{align*}
  Because
  \begin{align*}
  \PP( \limsup B_k) &= \underset{n \to +\infty}{\lim} \PP \left( \underset{k\geq n}{\bigcup} B_k \right)\\
   &\geq \underset{n \to +\infty}{\lim} \underset{ k \geq n}{\inf} \PP(B_k)\\
   & > 0
   \end{align*}
  we obtain the result.
\end{proof}

\subsubsection{The supersolution}

On the event $B_k$ we construct a supersolution of the Hamilton-Jacobi equation which controls the growth rate up to time $T_k$. Since we assume $B_k$ occurs this means that there is a horizontal rectangle of length $\lambda T_k +1$ centered at a point $X = (X_1,X_2)$ with $|X| \leq {\delta} T_k$.

Define $\sigma_1 := \frac{\lambda T_k}{\sqrt{2}}$ and $\sigma_2 := \frac{\mu}{\sqrt{2}}$. Recall that $\lambda T_k+1$ and $\mu$ are the dimensions of the complete horizontal rectangle and that we have
$\lambda, \mu> \frac{16}{\eta}$ (see (\ref{cond_eta})).

Let us now define
$$h_i : x \in \RR \mapsto (x - X_i)(\mathbb{P}(P_i < x) - \mathbb{P}(P_i > x)) \; \text{ for } \; i= 1,2$$
where $P_i$ follows the normal distribution $ \mathcal{N}(X_i,\sigma_i^2)$ for  $i=1,2$. Now define $u^+ \, : \, (0,T_k) \times \RR^2 \to \RR$ by
\[
  u^+(t,x) :=  t(-\frac{1}{2}+\eta) + 2t\frac{h_1(x_1)}{\lambda T_k} + h_2(x_2).
\]
We aim to prove that $u$ is a supersolution on $t\in(0,T_k)$:
$$\partial_t \usup + H( \nabla \usup , x, \omega) -\Delta \usup \geq 0.$$

Let us motivate briefly the choice of $u^+$. Think of $\eta$ as being small.
In the game-theoretic interpretation, a supersolution is an upper bound on the value of the differential game described in Section \ref{game}. Thus, considering the strategy of Player 1 that pushes the state to the complete horizontal rectangle and keeps in it as often as possible, and estimating its cost gives an idea of what the supersolution should be.
Under this strategy, the state is most of the time in the rectangle, where the cost is $-1/2$. This explains the term $t(-1/2+\eta)$. The second term reflects the fact that the strategy of Player 2, pushing outward horizontally from the center of the rectangle, and the Brownian motion may make the state leave the rectangle by one of its lateral sides. Thus, this term takes into account the fact that while the state is out, the cost rises. If the Brownian motion makes the state leave the rectangle by one of its vertical edges, Player 1 will push the state back in the direction of the rectangle. This extra cost is contained inside $\eta$. Further, the game may begin near the rectangle and not inside, so Player 1 pushes the state to the rectangle. The third term reflects the cost of this action.

Let us now check formally that $u^+$ is indeed a supersolution. We have
\begin{align*}
\nabla \usup (t,x)&=\left(2t\frac{h_1'(x_1)}{\lambda T_k},h_2'(x_2)\right) \\
  \partial_{t} \usup (t,x) &= -\frac{1}{2} + \eta + 2\frac{h_1(x_1)}{\lambda T_k}
 \end{align*}
 and
$$\Delta \usup (t,x) =  2t\frac{h_1''(x_1)}{\lambda T_k} +  h_2''(x_2). $$
We have
\begin{align*}
    h_i'(x) = (\mathbb{P}(P_i < x) - \mathbb{P}(P_i > x)) + 2(x-X_i)f_{P_i}(x)  \\
    h_i''(x) = 2f_{P_i}(x)\left(2 - \frac{(x-X_i)^2}{\sigma_i^2}\right).
\end{align*}
Where $f_{P_i}$ is the density of $P_i$
$$f_{P_i}(x) = \frac{1}{\sqrt{2\pi}\sigma_i} \exp\left(-\frac{(x-X_i)^2}{2\sigma_i^2} \right).$$
Note that $ h_i'$ decreases on $(-\infty,X_i -\sqrt{2}\sigma_i)$, increases on
$(X_i -\sqrt{2}\sigma_i,X_i+\sqrt{2}\sigma_i)$, decreases on $(X_i+\sqrt{2}\sigma_i,+\infty)$, and attains its minimum and maximum respectively in $X_i -\sqrt{2}\sigma_i$ and
$X_i+\sqrt{2}\sigma_i$. It follows that for all $x \in \RR$,
\begin{equation*}
 | h_i'(x) | \leq \max\left(|h_i' (X_i-\sqrt{2}\sig_i)|,|h_i' (X_i+\sqrt{2}\sig_i)|\right) \leq 2.
\end{equation*}
Since $\displaystyle \lim_{x \rightarrow-\infty} h'_i(x)=-1$ and $\displaystyle \lim_{x \rightarrow+\infty} h'_i(x)=1$, when $|x - X_i| > \sqrt{2}\sigma_i $ we have
$$| h_i'(x) | \geq 1.$$
Last, note that for all $x \in \RR$, $h_i(x) \geq 0$, and when $|x_1-X_1|>\sqrt{2} \sigma_1=\lambda T_k$,
\begin{align*}
  h_1(x_1)&=  (x_1-X_1) \left( \PP(P_1 < x_1) - \PP(P_1 > x_1) \right)\\
  &\geq  |x_1-X_1| (1 - \PP(|P_1 -X_1| > \lambda T_k)) \\
  & \geq \lambda T_k \left(1 - \frac{\sig_1^2}{\lambda^2 T_k^2} \right) \\
  & \geq \frac{\lambda T_k}{2}.
\end{align*}

Thus, the following hold for all $x \in \RR^2$:
\begin{align}
 \partial_t \usup &\geq -\frac{1}{2}+\eta
 \\
 | \partial_{x_1} \usup (t,x) | &\leq \frac{4}{\lambda}
 \\
  \partial^2_{x_1} \usup (t,x) &\leq \frac{8}{\lambda \sqrt{2 \pi} \sigma_1} \leq \frac{4}{\lambda}
  \\
    \partial^2_{x_2} \usup (t,x) &
   \leq \frac{4}{\sqrt{2\pi}\sigma_2} \leq \frac{4}{\mu}
 \end{align}
Moreover, when $|x_1-X_1|>\sqrt{2} \sigma_1=\lambda T_k$, we have
\begin{align*}
 \partial_t \usup &\geq \frac{1}{2}+\eta
 \\
 \partial^2_{x_1} \usup (t,x) &\leq 0
\end{align*}
Last, when $|x_2-X_2|>\sqrt{2} \sigma_2$, we have
\begin{align*}
 \partial^2_{x_2} \usup (t,x) &\leq 0
 \\
 | \partial_{x_2} \usup | &\geq 1
\end{align*}
By the above properties, $||\nabla \usup ||_\infty \leq 2$ and $H(x,\nabla u^+, \omega)  = -c_\omega(x) + |\partial_{x_2} \usup | - |\partial_{x_1} \usup |$.

 Let $(t,x) \in (0,T_k) \times  \RR^2$. We distinguish between the following cases:

 \begin{itemize}
   \item \textbf{Case 1 : $|x_2-X_2| > \sqrt{2} \sigma_2 $ and $|x_1 - X_1 | > \sqrt{2} \sigma_1$ :}

We have
\begin{align*}
  \partial_t \usup + H( \nabla \usup , x, \omega) -\Delta \usup & \geq \left(\frac{1}{2} + \eta  \right) + \left(  -\frac{1}{2} + 0 - \frac{4}{\lambda} \right)  -0 \geq 0
  \end{align*}

\item \textbf{Case 2 : $|x_2-X_2| > \sqrt{2} \sigma_2 $ and $|x_1 - X_1 | < \sqrt{2} \sigma_1$ :}

We have
\begin{align*}
  \partial_t \usup + H( \nabla \usup , x, \omega) -\Delta \usup & \geq
  \left(-\frac{1}{2} + \eta\right) + \left(-\frac{1}{2} +  1 -\frac{4}{\lambda}\right) -\left(0+\frac{4}{\lambda}\right) \geq 0
\end{align*}

\item \textbf{Case 3 : $|x_2-X_2| < \sqrt{2} \sigma_2 $ and $|x_1 - X_1 | > \sqrt{2} \sigma_1$ :}

We have
\begin{align*}
  \partial_t \usup + H( \nabla \usup , x, \omega) -\Delta \usup  & \geq \left(\frac{1}{2} + \eta \right) + \left( -\frac{1}{2} +  0 -\frac{4}{\lambda} \right) -\frac{4}{\mu} \geq 0
  \end{align*}
\item \textbf{Case 4 : $|x_2-X_2| < \sqrt{2} \sigma_2 $ and $|x_1 - X_1 | < \sqrt{2} \sigma_1$ :}
In this case, we have $c_\omega(x) = -\frac{1}{2}$ because $x$ is in a complete horizontal rectangle.
Thus
\begin{align*}
  \partial_t \usup + H( \nabla \usup , x, \omega) -\Delta \usup  & \geq \left(-\frac{1}{2} + \eta \right) + \left( \frac{1}{2} + 0- \frac{4}{\lambda} \right) -\frac{4}{\lambda}-\frac{4}{\mu} \geq 0
  \end{align*}
\end{itemize}
Consequently $\usup$ is a supersolution. The comparison principle implies that for all $\omega \in \Omega_{1}$ we have
$$u(T_k,0,\omega) \leq \usup(T_k,0) \leq  T_k(-\frac{1}{2}+\eta) + 2T_k\frac{|X_1|}{\lambda T_k} + |X_2| \leq  T_k(-\frac{1}{2}+\eta) + T_k\frac{\lfloor {\delta} T_k \rfloor}{\lambda T_k} + \lfloor {\delta} T_k \rfloor .$$
Thus
$$\underset{T \to + \infty}{\liminf}\frac{1}{T}u(T,0,\omega) \leq \underset{k \to +\infty}{\lim} \frac{1}{T_k} \usup(T_k,0) = -\frac{1}{2} + \eta +\frac{{\delta}}{\lbd} + {\delta} \leq -\frac{1}{2} + \eta + 2{\delta}. $$
This is true for all ${\delta} > 0$, thus for all $\omega \in \Omega_{1}$ we have
$$\underset{T \to + \infty}{\liminf}\frac{1}{T}u(T,0,\omega) \leq -\frac{1}{2} + \eta .$$

\subsection{Subsolution of the Hamilton-Jacobi equation}

The construction of the subsolution is symmetrical to the one of the supersolution. For ${\delta}> 0$, we define $B'_k$ the event "there exists a center of a complete vertical rectangle of length $\lambda T_k +1$ at a distance smaller or equal to $\lfloor {\delta} T_k \rfloor$ from the origin." There exists a positive probability event $\Omega_2 \subset \Omega$ such that for all $\omega \in \Omega_2$, the events $(B'_k)_{k\geq 1 }$ occur infinitely often.
Because the map $c_\omega$ is symmetric the proof is the same as the supersolution case.

We can make a bijection between the sets $\Omega_1$ and $\Omega_2$. For each event $\omega_1$ in $\Omega_1$ we associate an event $\omega_2$ in $\Omega_2$ where the map is the one of $\omega_1$ turning of $90^{\circ}$ around the origin and inverting the values on the rectangles. This event is in $\Omega_2$ because the complete horizontal rectangle at a distance $\lfloor {\delta} T_k \rfloor$ of the origin becomes a complete vertical rectangle at the same distance of the origin. Note that the map is its own inverse.

Let us denote $\widehat{x} := (x_2,-x_1)$ for $x =(x_1,x_2) \in \RR^2$. Because $c_{\omega_2}(x) = -c_{\omega_1}(\widehat{x})$, we have for $||p||_\infty < 2$
$$H(x,p,\omega_2) = c_{\omega_2}(x) + |p_2| - |p_1| = - \left( c_{\omega_1}(\widehat{x}) + |p_1| - |p_2| \right)= - H(\widehat{x},\widehat{p},\omega_1).$$
If $X(\omega_2)$ is the center of complete vertical rectangle near the origin, we have $X(\omega_2) = \widehat{X(\omega_1)}$ where $X(\omega_1)$ is the center of the horizontal rectangle in the associate map of $\Omega_1$.

So if $u^+(t,x)$ is a supersolution on  $\Omega_1$ then $ u^-(t,x) := -u^+(t,\widehat{x})$ is a subsolution on $\Omega_2$. Indeed, $ \widehat{\nabla  u^- }(t,x) = -\nabla u^+(t,\widehat{x})$, $\Delta u^-(t,x) = -\Delta u^+(t,\widehat{x})$ and if $ \omega_2 \in \Omega_2$ there exists $\omega_1 \in \Omega_1$ such that
\begin{align*}
\partial_t u^-(t,x) + H(x,\nabla u^-(t,x),\omega_2) - \Delta u^-(t,x) &= - \left[ \partial_t u^+(t,\widehat{x}) + H(\widehat{x},\widehat{\nabla u^-}(t,x),\omega_1) - \Delta u^+(t,\widehat{x}) \right] \\
&= - \left[ \partial_t u^+(t,\widehat{x}) + H(\widehat{x},\nabla u^+(t,\widehat{x}),\omega_1) - \Delta u^+(t,\widehat{x}) \right] \\
&\leq 0.
\end{align*}
Thus for all $\omega \in \Omega_{2}$ we have
$$\underset{T \to + \infty}{\limsup}\frac{1}{T}u(T,0,\omega) \geq \frac{1}{2} -\eta.$$

\subsection*{Acknowledgments}
  The first author was partially supported by the National Science Foundation RTG grant DMS-1246999. The third author benefited from the support of the FMJH Program PGMO VarPDEMFG (ANR-11-LABX-0056-LMH) and from the support of EDF, Thales, Orange and Criteo. The authors would also like to thank Pierre Cardaliaguet, Sylvain Sorin and Takis Souganidis for helpful advice.

  \bibliographystyle{plain}
\bibliography{HJarticles}

{\it E-mail address:} \href{mailto:feldman@math.uchicago.edu}{\nolinkurl{feldman@math.uchicago.edu}}

\smallskip

\small{\sc Department of Mathematics, The University of Chicago, Chicago, IL 60637, USA}

\medskip

{\it E-mail address:} \href{mailto:jean-baptiste.fermanian@ens-rennes.fr}{\nolinkurl{jean-baptiste.fermanian@ens-rennes.fr}}

\smallskip

\small{\sc ENS Rennes, 35170 Bruz, France}

\medskip

{\it E-mail address:} \href{mailto:ziliotto@math.cnrs.fr}{\nolinkurl{ziliotto@math.cnrs.fr}}

\smallskip

\small{\sc CEREMADE, CNRS, Universit\'e Paris Dauphine, PSL Research Institute, Paris, France}

\end{document}